\documentclass[12pt,cd]{amsart}
\usepackage{latexsym}
\usepackage{amsmath,amscd,amssymb}
\usepackage{graphics}
 2

\usepackage{hyperref}
\newtheorem{theorem}{Theorem}[section]
\newtheorem{proposition}{Proposition}[section]

\newtheorem{definition}[theorem]{Definition}

\newtheorem{remark}[theorem]{Remark}
\numberwithin{equation}{section}

\newcommand{\mbr}{{\mathbb R}}

\newcommand{\mM}{{\mathcal{M}}}

\newcommand{\mF}{{\mathcal {F}}}

\title{Darboux chart on Projective limit of weak symplectic Banach manifold}
\author{Pradip Kumar }
\address{Department of Mathematics,
 Harish Chandra Research Institute\\ Allahabad 211019, Uttarpradesh\\ India} \email{pmishra@hri.res.in}
\begin{document}

\maketitle
\begin{abstract}
Suppose $M$ be the projective limit of weak symplectic
Banach manifolds $\{(M_i,\phi_{ij})\}_{i,j\in\mathbb N}$, where $M_i$ are modeled over reflexive Banach space and $\sigma$ is  compatible with the projective system (defined in the article).
We associate to each point $x\in M$, a Fr\'{e}chet space $H_x$.   We prove that
if $H_x$ are locally identical, then with certain smoothness and  bounded-ness condition, there exists a Darboux chart for the weak symplectic structure.
\end{abstract}
\section{Introduction}
 We say $M$ is the projective limit of Banach manifolds
(PLB-Manifold) modeled on a PLB-space $E=\varprojlim E_i$
if we have followings.
\begin{enumerate}
\item There is a projective system of Banach manifolds $\{M_i, \phi_{ji}\}_{i,j\in\mathbb N}$ such that $M= \varprojlim M_i$.
\item For each $p\in M$, we have $p= (p_i)$. $p_i\in M_i$, and there is a chart $(U_i,\psi_i)$ of $p_i\in M_i$ such that
    \begin{enumerate}
    \item $\phi_{ji}(U_j)\subset U_i$, $j\geq i$.
    \item  Let $\{E_i,\rho_{ji}\}_{i,j\in\mathbb N}$ be a projective systems of Banach spaces, where each $\rho_{ji}$ is
    inclusion map
     and the diagram $$\begin{CD}
U_j @>\psi_j>> \psi_j(U_j)\\
@VV\phi_{ji}V @VV\rho_{ji}V\\
U_i @>\psi_i>> \psi_i(U_i)
\end{CD}$$
 commutes.
\item $\varprojlim\psi_i(U_i)$ is open in $E$ and $\varprojlim U_i$ is open in $M$ with the inverse limit topology.
\end{enumerate}
\end{enumerate}
The space $M$ satisfying the above properties has a natural Fr\'{e}chet manifold structure.
 The differential structure on $M$ is determined by the co-ordinate map
 $\psi: U=\varprojlim U_i\to \psi(U)= \varprojlim\psi_i(U_i)$. Therefore a smooth structure on these type of manifolds is completely determined by the smooth structure on the sequence.  George Galanis in \cite{galanis1,galanis2,galanis3,galanis4} has studied this type of manifolds and smooth structure on it.

G. Galanis, in a series of articles \cite{galanis1,galanis2,galanis3,galanis4} discussed various properties of PLB-manifolds.

Suppose $M$ is endowed with a weak symplectic structure $\sigma$ and each $M_i$ is endowed with a weak symplectic structure $\sigma^i$.  In section 3, we will define the notion of compatible symplectic structure on PLB manifold with projective symstem.

A 2 form $\sigma\in \Omega^2(\mM)$
is called a weak symplectic form on $\mM$
if it is closed ($d\sigma=0$) and if
the associated vector bundle homomorphism
$\sigma^b: T\mM\to T^*\mM$ is injective. A 2 form $\sigma$ on a Banach manifold is called a strong symplectic on
$\mM$ if it is closed ($d\sigma=0)$ and its associated
vector bundle homomorphism
$\sigma^b: T\mM\to T^*\mM$ is invertible
with smooth inverse.

In the  case of strong symplectic Banach manifold, the vector bundle $TM$ has reflexive fibers $T_x\mM$.

For a Fr\'{e}chet manifold with strong topology
$T_x^*\mM$ is not topologically isomorphic to $
T_x\mM$.  Hence in PLB manifold we  have only weak symplectic structure.

In 1969, for a strong symplectic Banach manifold,  Weinstein \cite{weinstein} proved that there exists Darboux chart.  In 1972, Marsden \cite{marseden} showed that the Darboux theorem fails for a weak symplectic Banach manifold.  In 1999 Bambusi \cite{bambusi} gave a necessary and sufficient condition for existence of  Darboux charts for weak symplectic Banach manifold (in the case when model space are reflexive).

In this article, we found sufficient  conditions such that a weak symplectic structure on a PLB- manifolds admit Darboux chart.

Our idea is similar to the idea of Weinstein\cite{weinstein} and  Bambusi\cite{bambusi}. In this article we  develop and extend the main ideas of Bambusi\cite{bambusi} in the context of PLB-manifold.   This article is arranged as follow.

In section 2, we will discuss PLB-manifolds and in section 3,  we will define the compatible symplectic structure. In section 5,  we will give condition for the existence of Darboux chart in PLB-manifolds whose symplectic structure is compatible with the projective system.

\textbf{Acknowledgement:} I wish to thank Prof. George N. Galanis for his continuous interest in this work and his valuable suggestions and providing the references for completing this article.

\section{Preliminaries}
In this section we will revise the basics about PLB manifold. Basic definitions which will be used in later section is given in this section.

\subsection{Smooth maps on PLB space}\label{sec:smooth PLB space}
Suppose $\{E_i,\phi_{ji}\}_{i,j\in\mathbb N}$ and
$\{F_i,\rho_{ji}\}_{i,j\in\mathbb N}$ be the
projective system of Banach spaces and
$E= \varprojlim E_i$ and $F=\varprojlim F_i$.  $E$ and
$F$ are Fr\'{e}chet space.

\begin{definition}\label{defn:projective system of mapping}[Projective system of mapping].
We say $\{f_i: E_i\to F_i\}_{i\in\mathbb N}$ is a projective system of mapping if the following diagram commutes.
     $$\begin{CD}
E_j @>f_j>> F_j\\
@VV\phi_{ji}V @VV\rho_{ji}V\\
E_i @>f_i>> F_i.
\end{CD}$$
\end{definition}
We denote the canonical mapping of $E\to E_i$ by $e_i$ and $F\to F_i$ by $e_i'$.
\begin{definition}\label{defn:the projective limit of a system}
We say $f: E\to F$ is the projective limit of a system $\{f_i: E_i\to F_i\}_{i\in\mathbb N}$ if for each $i$,
the following diagram commutes.
     $$\begin{CD}
E_i @>f_i>> F_i\\
@VVe_iV @VVe_i'V\\
E @>f>> F.
\end{CD}$$
\end{definition}
We define the map $\varprojlim f_i$ as following.

If $\{f_i\}$ is the projective system of mappings
then we see that for any $x=(x_i)\in \varprojlim E_i$,
$(f_i(x_i))\in\varprojlim F_i$.
We define $(\varprojlim f_i)(x)= (f_i(x_i))\in F$.
If $f$ be the projective limit  of system $\{f_i\}$ then we have
$$f(x)= (f_i(x_i)).$$
We denote the projective limit of system $\{f_i\}$ by $\varprojlim f_i$.
Also $\left(\varprojlim f_i\right)(x)= (f_i(x_i))= f(x)$.

G. Galanis in \cite{galanis1} has given the following criterion for checking smoothness of the map $f: E\to F$ such
that $f=\varprojlim f_i.$

\begin{theorem}[lemma 1.2,\cite{galanis1}]\label{thm:galsmoothness}
Suppose $E= \varprojlim E_i$ and $F=\varprojlim F_i$ and $\{f_i: E_i\to F_i\}_{i\in\mathbb N}$ be a projective system of
smooth mapping then the following  holds.
\begin{enumerate}
\item $f$ is $C^\infty$, in the sense of J. Leslie.\cite{Leslie}
\item $df(x)= \varprojlim df_i(x_i)$, $x=(x_i)\in E$.
\item $df= \varprojlim df_i$.
\end{enumerate}
\end{theorem}

We mention here that in the PLB-space if a map $f=\varprojlim f_i$ as in theorem \ref{thm:galsmoothness} is smooth in the sense of J. Leslie then it will be smooth in the sense of Kriegl and Michor too.  This can be seen as following.

Let $c:\mbr\to \varprojlim U_i\subset E$ be a smooth curve and $f:=\varprojlim f_i:  \varprojlim U_i\to E$
be a map which is smooth in the sense of J. Leslie (i.e $f$ satisfies theorem \ref{thm:galsmoothness}).
We can identify $c(t)$ as  $c(t)= (c_i(t))$ where $\phi_{ji}(c_j(t))= c_i(t)$.  $c$ is smooth if and only each
$c_i$ is smooth (here we are using the fact that $\pi_i: \varprojlim E_i\to E_i$ is a smooth map and $c_i= \pi_i\circ c$).

Now let $\tilde{c}: \mbr\to \varprojlim U_i$,
defined by
$\tilde{c}(t)= f\circ c(t)$,  we see that $$\tilde{c}(t)= \varprojlim(f_i\circ c_i)(t)$$
As each $f_i\circ c_i$ is smooth and hence the derivative of every order exists.
Therefore by theorem \ref{thm:galsmoothness}, $\tilde{c}$ is smooth in the sense of J. Leslie.
Recall that Smoothness of curves is defined by in the same way by J. Leslie and by Kriegl and Michor.
 This proves $f\circ c$ is smooth curve for every smooth curve $c$.

Therefore $f$ defined as in theorem \ref{thm:galsmoothness} is smooth in the sense of Kriegl and Michor too.

In view of  above discussion,
we see that the calculus on
PLB-manifolds agrees with the Kriegl and Michor calculus.

In particular we have the following proposition.
\begin{proposition}\label{prop:KM_and_Galanis}
Let $E$ be a Fr\'{e}chet space and $E= \varprojlim E_i$,
$E_is$ are Banach spaces.
Co-ordinate map $\psi_i$ as defined
above,  $\psi_i: U= \varprojlim U_i\to \psi(U)= \varprojlim\psi_i(U_i)$
is smooth in the sense of Kriegl and Michor \cite{KM}.
\end{proposition}
\begin{proof}
We have to check that $c:\mbr\to \varprojlim U_i$ is smooth if and only if $\tilde{c}: \mbr \to \varprojlim E_i$ defined by
$$\tilde{c}(t):=(\varprojlim\psi_i)\circ c (t)= \varprojlim(\psi_i\circ c_i)(t)$$ is smooth.
 But this follows by remark after the theorem \ref{thm:galsmoothness}.
 \end{proof}

Summarizing above, PLB manifolds are
Fr\'{e}chet manifolds and fit with Kriegl and Michor calculus.

\subsubsection{Tangent bundle of a PLB manifold}\label{subsec,tgnt bndle of PLB mnfld}
We will follow \cite{galanis5} for the following discussion on the tangent bundle of a PLB manifold.  If $\{M_i,\phi_{ji}\}$
be a projective system of Banach manifolds and
$M=\varprojlim M_i$ is the $PLB$ manifold for
this projective system.
For $p\in M=\varprojlim M_i$, we have $p =(p_i)$.
We observe that $\{T_{p_j} M_j, T_{p_j}\phi_{ji}\}$ is a
projective system of topological vector spaces.
The identification $T_pM\simeq \varprojlim T_{p_i} M_i$ is given by the mapping $h:= \varprojlim T_p\phi_i$,
where $\phi_i$ are the canonical projection of $M$.
We refer to \cite{galanis5} for the proof.

Galanis in \cite{galanis5} proved that  $\{TM_i, T\phi_{ji}\}$ is a projective system of Banach manifolds and  $TM$ is a PLB manifold and
$$TM\simeq \varprojlim TM_i\text{ by } g= \varprojlim T\phi_i.$$

\subsection{Flow of a vector field:}
Suppose $Y:U \to E$ be a vector field on open subset $U$ of a Fr\'{e}chet space $E$.
For ensuring existence of the flow of a vector field on a Fr\'{e}chet space either
we can go the direction set by H. Omri as in \cite{omri} or
we can demand for some kind of tame condition as in  \cite{Hamilton}. In general Fr\'{e}chet space, flow of
a vector field may not exists for example see \cite{KM}.

Situation become easier while working on some special
vector fields on $PLB$-manifolds.  We start by following definition.
\begin{definition}[$\mu$-Lipschitz map] Let $E= \varprojlim E_i$ be a  Fr\'{e}chet space.
A mapping $\phi: E\to E $ is called projective $\mu$-Lipschitz ($\mu$, a positive real number) if
there are $\phi_i: E_i\to E_i$ such that
$\phi= \varprojlim \phi_i$ and  for every $i$,
$\phi_i$ is a $\mu$ Lipschitz map on each $E_i$.

Suppose $M$ and $N$ be PLB manifolds modeled over a PLB space $E$.
We say $\phi= M\to N$ is locally $\mu$-Lipsctiz map if there exists a coordinate charts around $p$ and
$f(p)$ such that on that chart $\phi $ is projective $\mu$-Lipscitz.
\end{definition}
In some other co-ordinate chart $\phi$ may not be
 locally $\mu$ Lipscitz.  We make a remark here
 that $\mu$ is fixed for every $i$.  \
In fact if each $\phi_i$ are smooth map on
Banach space then $\phi_i$ are locally $\mu_i$ Lipschitz for some positive
real number $\mu_i$
but above definition demands that $\phi_i$ are  $\mu$ Lipschitz.

We mention a theorem by G. Galanis which
we will be using to determine the existence of flow of the some special vector fields.
\begin{theorem}\label{thm:IFT} [Theorem 3, \cite{galanis4}]
Let $E= \varprojlim E_i$ be a Fr\'{e}chet space
\begin{equation}\label{de1}x'= \phi(t,x)\end{equation} a differential equation on $E$, where
$\phi=\mathbb R\times E\to E$ is a projective $\mu$-Lipschitz mapping.
Assume that  for an initial point $(t_0,x_0)\in \mathbb R\times E$ there exists a constant $\tau\in\mathbb R$ such that
$$M_1:= \sup\{p_i(\phi(t,x_0)): i\in \mathbb N, t\in[t_0-\tau, t_0+\tau]\}<+\infty$$
then this differential equation has a unique solution in the interval
$I=[t_0-a,t_0+a]$, where $a=\inf\{\tau, \frac{1}{M_1+\mu}\}$
\end{theorem}

In above theorem, we used the notation of article \cite{galanis4}.

Galanis and Palamides proved above theorem in \cite{galanis4}. Following we sketch the proof of above
theorem.
\begin{proof}
\begin{enumerate}
\item As $\phi:\mbr\times E\to E$ is a projective $\mu$-Lipschitz map, there is a family $\{\phi_i:\mbr\times E_i\to E_i\}_{i\in\mathbb N}$ of
$\mu$- Lipschitz map realizing $\phi$.  Corresponding to system as
given in equation \ref{de1}, we have system of ordinary
differential equations on the Banach space $E_i$ defined:
\begin{equation}\label{de}x_i'=\phi_i(t, x_i)\end{equation}
\item With the condition given in the theorem, Galanis (\cite{galanis4})  proved that there is a unique solution $x_i$ for
each  equation \ref{de} defined on the interval $I= [t_0-a, t_0+a]$.
\item These solution are related that is we have $\rho_{ji}\circ x_j= x_i$.
\item As a result we can define $x= \varprojlim x_i$.
By criterion of smooth map as in section \ref{sec:smooth PLB space}, if
each $x_i$ is smooth then $x$ is smooth map defined on $I$.
\item Using the uniqueness of each $x_i$
 with respect to the initial data,
Prof. Galanis \cite{galanis4} proved that $x$ as defined above is unique with respect to the initial data $x(t_0)= x_0$.
\end{enumerate}
\end{proof}
\begin{theorem}\label{Flow}
Let $X: E \to E$ be a smooth vector field on $E$.
Suppose $X$ is a $\mu$-Lipschitz projective  map such that
$X= \varprojlim X_i$.
Each $X_i$ is defined on $E_i$ and for each $i$ and
$x_i\in E_i$, $\|X_i(x_i)\|_i<M$ for some fixed $M>0$.

Then there exists an $\epsilon>0$ such that for every
$y\in E$ there is a unique integral curve
$t\to x^y(t) \in E $  such that $x^y(0)= y$. Also the map
$F: [-\epsilon, \epsilon] \times E\to E$ defined by
$$F(t, p):= x_p(t)$$ is a smooth map.
\end{theorem}
Technique of the proof given below is taken from \cite{galanis4} and \cite{Cartan}.
\begin{proof}
Let $y\in E$ and consider the differential equation
$$x'(t)= X(x(t))$$ with  $x(0)=y$.  For any $t\geq 0 $, we have
$$p_i(X(x(t)))= \|X_i(x_i(t))\|_i\leq M.$$
Let $\epsilon= \frac{1}{M+\mu}$, then theorem \ref{thm:IFT}, implies that exists a unique curve
$$x^y: [-\epsilon, \epsilon]\to E$$
such that $x^y(0)= y$.
In fact $x^y= \varprojlim x_i^{y_i}$ ($\varprojlim x_i^{y_i}$
is well defined, we refer page 5 \cite{galanis4})
where $x_i^{y_i}: [-\epsilon, \epsilon]\to E_i$ be
the integral curves passing through $y_i$ of
corresponding vector fields
$X_i$.

For each $i$ define
$F_i: [-\epsilon, \epsilon]\times E_i\to E_i$ such that
$F_i(t,p_i)= x_i^{p_i}(t)$ and
$F: [-\epsilon, \epsilon]\times E\to E$ such that
$F(t,p)= x^{p}(t)$.

Each $F_i$ are smooth function and $\{F_i\}$ makes
projective system of map with projective limit $F$.  That is we have
$F= \varprojlim F_i$ and this proves $F$ is a smooth map.
\end{proof}

%%%%%%%%%%%%%%%%%%%%%%%%%%%%%%%%%%%%%%%%%%%%%%%%%%%%%%%%%
\section{Compatible symplectic structure}

Suppose $(M,\sigma)$ be a PLB manifold with a weak symplectic form $\sigma$.   In this section, we will define  compatibility of $\sigma$ with the projective system.

\subsection{Basics about weak symplectic structure}\label{sec:basics about weak symplectic str}
Let $M$ be a PLB manifold and $\{M_i,\phi_{ji}\}_{i,j\in\mathbb N}$ be a projective system of Banach manifolds with $M= \varprojlim M_i$. Suppose each $M_i$ is modeled over a reflexive Banach space $E_i$ and each $M_i$ has a weak symplectic structure $\sigma^i$.

Let $x\in M$, we have $x= (x_i)$ where $ \phi_{ji}(x_j)= x_i$.  For each $x_i$, following \cite{bambusi}, we define a norm on $T_{x_i} M_i$, for $X\in T_{x_i}M_i$,
$$\|X\|_{\mathcal{F}_{x_i}}:= \sup_{\|Y\|_i=1}|\sigma^i_{x_i}(X,Y)|$$

where $\|.\|_i$ is the norm on the Banach space $T_{x_i}M_i$. Let $\mathcal{F}_{x_i}$ be the completion
of $T_{x_i}M_i$ with respect to the $\|.\|_{\mF_{x_i}}$ norm. As each $T_{x_i}M_i$ is a reflexive Banach space, we have that the induced map (Lemma 2.8,\cite{bambusi}),
$$(\sigma_{x_i}^i)^b: T_{x_i}M_i\to \mathcal{F}_{x_i}^*;\;\; X\to \sigma^i_{x_i}(X,.)$$
is a topological isomorphism.

Define for each $i, j\in\mathbb N\;\; j\geq i$ and given $x=(x_i)\in \varprojlim M_i= M$,
$$\psi_{ji}:\mathcal{F}_{x_j}^*\to \mathcal{F}_{x_i}^*\text{ by }\;$$
\begin{equation}\label{psiji}\psi_{ji}= (\sigma^i_{x_i})^b\circ T_{x_j}\phi_{ji}\circ((\sigma^j_{x_j})^b)^{-1}
\end{equation}
where $T_{x_j}\phi_{ji}$ is the differential of $\phi_{ji}: M_j\to M_i$ at $x_j$.  We see that $\{\mathcal{F}_{x_i}^*,\psi_{ji}\}_{i, j\in\mathbb N}$ is a projective system of Banach spaces and smooth maps (since for any $k\geq j\geq i$, we have $\psi_{ki}=\psi_{kj}\circ\psi_{ji}$). We see that $\{(\sigma^i_{x_i})^b: T_{x_i}M_i\to \mathcal{F}_{x_i}^*\}$ is a projective system of mappings because $$(\sigma^j_{x_j})^b\circ \psi_{ji}= T_{x_j}\phi_{ji}\circ (\sigma^i_{x_i})^b.$$

Fix some point $p= (p_i)\in M$.   We know that for a fixed $p= (p_i)\in M$,  $\{T_{p_j} M_j, T_{p_j}\phi_{ji}\}_{i,j\in \mathbb N}$ is a projective system of Banach spaces. In section 1.3.3 we saw that $T_pM\simeq\varprojlim T_{p_j}M_j$. Let $h_p$ be the isomorphism from $T_pM \to \varprojlim T_{p_j}M_j$ as defined in \cite{galanis5}.

For some coordinate neighborhood  $U=\varprojlim U_i$ around $p$. Let $\sigma_p:= \sigma|_{x=p}$ be the constant symplectic structure on $U$ with a natural parallelism $TU\simeq U\times E$ where $E\simeq T_pM$ is a Fr\'{e}chet space.
On each $U_i$ we have a corresponding constant symplectic structure $\sigma_{p_i}^i$ with the natural parallelism.

For $t\in [-1,1]$ we define $\sigma^t:= \sigma_p+t(\sigma-\sigma_p)$ and similarly  $(\sigma^i)^t:=\sigma^i_{p_i}+t(\sigma^i-\sigma_{p_i}^i)$
Suppose for some $q= (q_i)\in M$ and for some $t\in [-1,1]$, $((\sigma_{q})^{tb})^{-1}$ and $(((\sigma^i_{q_i}))^{tb})^{-1}$ exist for each $i$.  Then $\{(((\sigma^i_{q_i}))^{tb})^{-1}:  \mathcal{F}_{q_i}^{*t}\to T_{q_i}M_i\}  $ is a projective system of map. Here $\mathcal{F}_{q_i}^{*t}$ defined in the same way as $\mathcal{F}_{q_i}^*$ above.   Where $\mathcal{F}_{q_i}^*$ are spaces corresponding to $\sigma^i$,  $\mathcal{F}_{q_i}^{*t}$ are spaces corresponding to weak symplectic structure $(\sigma^i)^t$.  Also for fixed $t$ corresponding to $\psi_{ji}$ map,  we have  map $\psi_{ji}^t$ (for the weak symplectic structure  $(\sigma^i)^t$). Therefore we see that for each $t$, the collections $\{(((\sigma^i_{q_i}))^{tb})^{-1}:  \mathcal{F}_{q_i}^{*t}\to T_{q_i}M_i\} $ is a projective system of mapping.

For a weak symplectic structure $\sigma$ on a PLB manifold manifold $M$ as discussed above, we have for each $p= (p_i)$, $\sigma_p^b$ is a map on $T_pM$ that is  $\sigma_p^b: T_pM\to T_p^*M$.  Let $h_p: T_pM\to \varprojlim T_{p_j}M_j$ be an isomorphism \cite{galanis5}.  With this identification we can consider $\sigma_p^b$ as a map defined on $\varprojlim T_{p_j}M_j$.

Now we are in a position  to define a compatible symplectic structure.
\subsection*{Compatible symplectic structure}
We say that a weak symplectic structure $\sigma$ on $M$ is compatible with the projective system if the following is satisfied:
\begin{enumerate}
\item Suppose there are weak symplectic structure $\sigma^i$ on $M_i$ such that for every $x\in M$,  $ \sigma^b_x:=\varprojlim (\sigma_{x_i}^i)^b$.
\item If there exists a 1 - form $\alpha$ such that for each $x\in M$ and $\alpha_x= (\alpha_{x_i}^i)\in \varprojlim \mathcal{F}_{x_i}^*$,  we must have $h_x((\sigma_{x})^{tb})^{-1} (\alpha_{x})=\left(((\sigma_{x_i}^{i})^{tb})^{-1} (\alpha^i_{x_i})\right)$ whenever defined.
\item For such $\alpha$ as in above, whenever  $Y^i_t(x_i):=[\left(x_i,((\sigma_{x_i}^{i})^{tb})^{-1} (\alpha^i_{x_i})\right)]$ is defined  on  some open set $U_i$ of $M_i$, it is defined on whole $M_i$.  Each $Y^i_t$ is $\mu$-Lipschitz  for some fixed positive real number $\mu$.
\end{enumerate}

\begin{remark}
In condition 1, $\varprojlim (\sigma^i_{x_i})^b$ make sense because $\{(\sigma^i_{x_i})^b\}$  makes projective system of mapping.   $\sigma_x^b$ as a map $ \sigma_x^b:\varprojlim T_{x_i}M_i\to \mathcal{F}_{x_i}^*$ is the projective limit of the maps $(\sigma^i_{x_i})^b$.

\end{remark}

\begin{remark}
For every $j\geq i$, condition 2 demands $$h_x((\sigma_{x})^{tb})^{-1} (\alpha_{x})=\left(((\sigma_{x_i}^{i})^{tb})^{-1} (\alpha^i_{x_i})\right)\in \varprojlim T_{x_i}M_i.$$  This means we have , $$T_{x_j}\phi_{ji}(((\sigma_{x_j}^{i})^{tb})^{-1} (\alpha^j_{x_j}))=((\sigma_{x_i}^{i})^{tb})^{-1} (\alpha^ i_{x_i}).$$
\end{remark}

The definition of compatibility of $\sigma$ arise while exploring the possibility of existence of Darboux chart for the case of the loop space $(LM, \Omega^\omega)$ discussed \cite{pradip2}.

\section{A Fr\'{e}chet space used in the theorem}
Since the Darboux theorem is a local result, we will work on some open subset $\mathcal{U}$ (containing zero) of a Fr\'{e}chet space $E$ which is the projective limit of Banach spaces. We have $\{E_i,\phi_{ji}\}$ is the  inverse system of Banach spaces (manifolds) $E_i$ and
$E= \varprojlim E_i$.    As $E$ is a PLB manifold, each $\phi_{ji}$ is the inclusion map.  We can assume that
$\mathcal{U}= \varprojlim\mathcal{U}_i$, where $\mathcal{U}_i$ are open subsets of Banach spaces $E_i$.  Let $\sigma$ be a compatible symplectic structure on the PLB manifold $\mathcal{U}$.

As we discussed earlier, for a fixed $x= (x_i)\in E$,  $\{T_{x_j} E_j= E_j, T_{x_j}\phi_{ji}=\phi_{ji}\}_{i,j\in \mathbb N}$ is a projective system of Banach spaces and $E=T_x E=\varprojlim T_{x_j}E_j$.

If the topology on $E$ is generated by the collection of semi-norms $\{p_k: k\in \mathbb N\}$ then for each $x\in E$ and
$k\in \mathbb N$, define norms on
$E$ as following.  For $X\in E$,
$$p_k^x(X):= \sup_{p_k(Y)=1} |\sigma_x(X,Y)|.$$
All $p_k^x$ are norms on $E$ and collection $\{p_k^x: k\in\mathbb{N}\}$ generate a topology on $E$.
Let completion of $E$ with respect to this collection be denoted by
$\mathcal{F}_x$. Then $\mathcal{F}_x$ is a  Fr\'{e}chet space.

For a fixed $x\in \mathcal{U}$, let  $H_x:=\{\sigma_x(X,.): X\in E\}$. We can extend $\sigma$ as a continuous bilinear map $E\times \mathcal{F}_x\to \mathbb R.$

%and therefore we can assume $H_x\subset \mathcal{F}_x^*$.
%For each $k\in\mathbb N$ and $f\in C^\infty(\mathcal{F}_x,\mbr)$, define the norms
%$$\|f\|_k = \sup_{\|X\|_k=1}|f(x)|. $$
%This collection of norms makes $C^\infty(\mathcal{F}_x,\mbr)$ as\ LCTV space.  Though $\mathcal{F}_x^*$ is not a Fr\'{e}chet space but $H_x$ is a
%Fr\'{e}chet space because it is a closed subspace of $C^\infty( \mathcal{F}_x,\mathbb R)$. In fact we have $H_x\subset \mF_x^*\subset C^\infty(\mathcal{F}_x,\mbr).$ We have  $H_x=\{\sigma_x(X,.): X\in E\}$ and $C^\infty(\mF_x,\mathbb R)$ are metric spaces. And if $X_n\to X$ in $E$, then $\sigma(X_n,.)\to \sigma(X,.)$ in $H_x$. Also in metric space the limit is unique.  Therefore if $\sigma_x(X_n,.)\to f\in C^\infty(F_x,\mathbb R)$ then there exists $X\in E$ such that
%$f= \sigma_x(X,.)$.

For $x= (x_i)$, in the section \ref{sec:basics about weak symplectic str} we defined $\mathcal{F}_{x_i}^*$.
If $\sigma$ on $E$ is compatible with the inductive maps,
for $X\in T_x E= E$
%observe that $\sigma_{x_i}(X_i,.) $ is an element of $\mathcal{F}_{x_i}^*$.  Now it is easy to see that  the map
$$\sigma_x(X,.)= (\sigma_{x_i}^i(X_i,.))\in \varprojlim \mF_{x_i}^*$$
Therefore $H_x \subset\varprojlim \mF_{x_i}^*$ as a set.

We know that each $(\sigma^i_{x_i})^b: T_{x_i}E_i(= E_i)\to \mathcal{F}_{x_i}^*$ is a topological isomorphism. Therefore a typical element of $\mathcal{F}_{x_i}^*$ will be given by $\sigma^i_{x_i}(X_i,.)$ for some $X_i\in E_i$.  Hence typical element of $\varprojlim \mF_{x_i}^*$ will be given by $(\sigma_{x_i}^i(X_i,.))$ where we must have $\psi_{ji}(\sigma_{x_j}^j(X_j,.))= \sigma_{x_i}^i(X_i,.)$.  This will happen if and only if $T_{x_j}\phi_{ji}(X_j)=\phi_{ji}(X_j)= X_i$. (since we can identify $(\sigma_{x_i}^i(X_i,.))=\sigma_x(X,.)$).  This shows that as set theoretic, $\varprojlim \mF_{x_i}^*\subset H_x$.

Therefore we have, as set, $\varprojlim \mF_{x_i}^*= H_x$.

On $H_x$, we have two possible topology:
\begin{enumerate}
\item Projective limit topology when we identify $H_x$ as $\varprojlim \mathcal{F}_{x_i}^*$. Projective limit topology on $H_x$ given as: For $i\in\mathbb N$, we have $\|\sigma_x(X,.)\|_i:= \|\sigma^i_{x_i}(X_i,.)\|_{op}$. Here at the right hand side of expression $\|.\|_{op}$ is operator norm  of $\sigma^i_{x_i}(X_i,.)$ as element of $\mathcal{F}_{x_i}^*$.
\item And induced topology when we consider $H_x$ as subset of $E^*$.
\end{enumerate}
We fix the notation $\|.\|_i$ for norm on Banach space $T_{x_i}E_i=E_i$ and we recall that $\|.\|_{\mathcal{F}_{x_i}}$ is norm on $\mathcal{F}_{x_i}$.

We have for $X,Y\in T_{x_i}E_i= E_i$,
\begin{equation}\label{sigmaXY}
\|X\|_i\|Y\|_{\mathcal{F}_{x_i}}\geq |\sigma_{x_i}^i(X,Y)|.
\end{equation}
Fix $\epsilon >0$, for $Y\in \mathcal{F}_{x_i}$ we have a sequence $(Y_n)\in E_i$ such that $\lim_{n\to \infty}Y_n= Y$ in $\mathcal{F}_{x_i}$. This mean there exists $N^Y\in \mathbb N$ such that
\begin{equation}\label{YNY}\|Y_n\|_{\mathcal{F}_{x_i}}-\epsilon <\|Y\|_{\mathcal{F}_{x_i}}<\|Y_n\|_{\mathcal{F}_{x_i}}+\epsilon;\;\; \forall n\geq N^Y.
\end{equation}
Let $A=\{Y\in \mathcal{F}_{x_i}: \|Y\|_{\mathcal{F}_{x_i}}=1\}$. For each $Y\in A$, there exists a sequence $(Y_n)$ in $E_i$ such that equation \ref{YNY} satisfies for some $N^Y$. We have a collection call $A^Y_\epsilon$ and $B^Y_\epsilon$ as follows:  For each $Y\in A$, fix a sequence $(Y_n)\subset E_i$ such that $\lim Y_n= Y$.
$$A^Y_\epsilon:= \{Y_n: n\geq N^Y\}\;\text{ and } B^Y_\epsilon:= \{y\in E_i : \|y\|_{\mathcal{F}_{x_i}}-\epsilon <\|Y\|_{\mathcal{F}_{x_i}}\}.$$
Let $A^\epsilon:= \bigcup_{Y\in A}A^Y_\epsilon$ and  $B^\epsilon= \bigcup_{Y\in A} B^Y_\epsilon.$

If $X\in E_i$ is fixed and $f$ is defined on $E_i$ such that $f(Y)= |\sigma^i_{x_i}(X,Y)|$ and we have continuous extension of $f$ to $\mathcal{F}_{x_i}$. Then we have
$$\sup_{Y\in A}f(Y)\leq \sup_{Y\in A^\epsilon}f(Y)$$
Also $A^\epsilon\subset B^\epsilon$, we have
$$\sup_{Y\in A} f(Y)\leq \sup_{Y\in B^\epsilon} f(Y).$$
On $B^\epsilon$, that is for $X\in E_i$ and $Y\in B^\epsilon$, we have
$|\sigma^i_{x_i}(X,Y)|\leq(1+\epsilon)|\|X\|_i.$
Therefore we have
\begin{eqnarray*}
\|\sigma^i_{x_i}(X,.)\|_{op}&= &\sup_{Y\in A}|\sigma^i_{x_i}(X,Y)|\\
&\leq& \sup_{Y\in B^\epsilon}|\sigma^i_{x_i}(X,Y)|\\
&\leq&(1+\epsilon)\|X\|_i.
\end{eqnarray*}
This is true for every $\epsilon$, we have for $X\in T_{x_i}E_i= E_i$,
\begin{equation}\label{alpha}
\|\sigma^i_{x_i}(X,.)\|_{op}\leq \|X\|_i.
\end{equation}

In the next section,  we will need the following proposition.
\begin{proposition}\label{prop:sigmaisomorpshim} $\sigma_x^b:E\to H_x$ is an isomorphism.
\end{proposition}
\begin{proof}
As discussed earlier,  for $x= (x_i)$, $H_x= \varprojlim \mathcal{F}_{x_i}^*$. For each $x_i$, we know $(\sigma^i_{x_i})^b:E_i (=T_{x_i}E_i)\to F^*_{x_i}$ is an isomorphism (lemma 2.8,\cite{bambusi}).  Let the inverse of $(\sigma^i_{x_i})^b$ is denoted by $J^i_{x_i}$.  Define
\begin{eqnarray*}&&\sigma^b_x: E=\varprojlim T_{x_i}E_i\to \varprojlim \mathcal{F}_{x_i}^*;\;\;\; \sigma_x^b(X)= \left((\sigma^i_{x_i})^b(X_i)\right)\text{  and }\\
&&J_x: H_x(=\varprojlim\mathcal{F}_{x_i}^*)\to \varprojlim  T_{x_i}E_i = E\text{ is defined as following.}
\end{eqnarray*}
If $\alpha_x\in H_x$, then
$J_x$ can be defined using $J_{x_i}^i= ((\sigma_{x_i}^i)^b)^{-1}$.
$$J_x(\alpha_x):= (J_{x_i}^i(\alpha_{x_i}))$$
As $\sigma$ is compatible with the inductive limit $(J^i_{x_i}(\alpha_{x_i}))\in \varprojlim T_{x_i} E_i= E$.

$J_x$ is the inverse of $\sigma_x^b$ and this gives an isomorphism.
\end{proof}
\begin{remark}
Form above proposition, it is clear that $H_x$ and $H_y$ are topologically isomorphic.
\end{remark}
\section{Condition for existence of Darboux chart on weak symplectic PLB manifolds}

We state the theorem for some open neighborhood of $0\in E$.

\begin{theorem}\label{mainthm}
Suppose
\begin{enumerate}
\item There exists a neighborhood $\mathcal{W}$ of $0\in E$, such that all $H_x$ are identical and $\sigma^{tb}_x: E\to H$ is isomorphism  for each $t$ and for each $x\in \mathcal{W}$.
\item There exists a vector field $X= (X_i)$ on $E$ support lies in $\mathcal{W}$ such that on $\mathcal{W}$,  $L_X\sigma= \sigma-\sigma_0$  and $X(0)=0$.
\item For every $i$ and $t\in [-1,1]$,  $X_i(x_i)$ as element on $E_i$ is bounded by $\frac{M}{\|((\sigma_{x_i}^i)^{tb})^{-1}\|_{op}}$ for some positive real $M$.
\item $\sigma: \mathcal{W} \to L(H, E)$, $x\to (\sigma^b_x)^{-1} $ is smooth.
\end{enumerate}
then there exists a coordinate chart $(\mathcal{V},\Phi)$ around zero
such that $\Phi^*\sigma= \sigma_0$.
\end{theorem}
For rest of discussion we denote $\overline{\sigma}= \sigma-\sigma_0$.  $\overline{\sigma}$ is defined on $\mathcal{W}$, an open neighborhood of 0.

\begin{remark}\label{remark}
Suppose $\sigma$ is a weak symplectic structure  as defined earlier (compatible with projective system) and there exists a vector field $X$ on $\mathcal{W}$ such that $L_X\sigma= \overline{\sigma}$.  Then we have $d(i_X\sigma)= \overline{\sigma}$.  Denote $\alpha= i_X\sigma= \sigma^b(X)$.  We define
$$\alpha_{x_i}^i:= (\sigma^i_{x_i})^b(X_i(x_i))$$
We can extend these $\alpha^i_{x_i}$ as $\alpha^i_{x_i}\in \mathcal{F}_{x_i}^*$. We see that $\psi_{ji}(\alpha^j_{x_j})= \alpha^i_{x_i}$.

%Also for a given $x= (x_i)$, $\{\alpha_{x_i}^i: E_i\to E_i^*\}$ is a projective system of mappings.
%Therefore we have$$\alpha_x(X)= \varprojlim \alpha_{x_i}^i(X_i)$$

Therefore $\varprojlim \alpha^i_{x_i}$ exists and as $\sigma$ is compatible with inductive maps, we have $$\alpha_x= \varprojlim(\alpha_{x_i}^i)$$
\end{remark}

\begin{remark} For any $x_i\in E_i$, $\alpha_{x_i}^i\in E_i^*$. As an element of $E_i^*$, we have
\begin{eqnarray*}\|\alpha^i_{x_i}\|_{op}&= & \sup_{\|Y\|_i=1}|\alpha^i_{x_i}(Y)|\\
&=&\sup_{\|Y\|_i=1}|\sigma^i_{x_i}(X_i(x_i),Y)|\\
&=& \|X_i(x_i)\|_{\mF_{x_i}}
\end{eqnarray*}
But as an an element of $\mathcal{F}_{x_i}^*$, by equation \ref{alpha}, we have
$$\|\alpha^i_{x_i}\|_{op}\leq \|X_i(x_i)\|_i$$
\end{remark}
\subsection{Proof of the theorem }

\begin{proof} There is an open neighborhood of $0\in E$,  $\mathcal{W}= \varprojlim \mathcal{W}_i$.  $0\in \varprojlim \mathcal{W}_i$ is identified with $0= (0,0,..)$.  On $\varprojlim \mathcal{W}_i$, define $\overline{\sigma}= \sigma_0-\sigma$ and  $\sigma^t= \sigma+t\overline{\sigma}$ for $t\in [-1,1,]$.

By remark \ref{remark}, we have on $\varprojlim \mathcal{W}'_i$, $\alpha_x= (\alpha_{x_i}^i)$ and  $$d(\alpha)= \overline{\sigma}.$$

As $\alpha = i_X\sigma$, we have $\alpha\in H$. We want to solve for $Y_t:\mathcal{W}\to E$ such that
$$i_{Y_t}\sigma^t= -\alpha$$
Consider $$(\sigma_x^t)^b: E\to H.$$  For $x\in \varprojlim \mathcal{W}_i$, $(\sigma_x^t)^b$ is isomorphism for all $t$. Hence for $x\in \varprojlim\mathcal{W}_i$,
$$Y_t(x)= ((\sigma_x^t)^b)^{-1}(\alpha_x)\;\;\; \text{ is well defined.}$$
%&&x\to \;\; \left((\sigma^b_x)^{-1}, \alpha_x\right)\to ((\sigma_x^t)^b)^{-1}(\alpha_x)
%\end{eqnarray*}
This is a smooth map.

We define  $Y_t^i: \mathcal{U}_i\to E_i$ such that
$$Y_t^i(x_i):= \left(((\sigma_{x_i}^i)^t)^b\right)^{-1}(\alpha^i_{x_i})$$
As $\sigma$ is compatible with inductive limits, we have $Y_t(x)= (Y_t^i(x_i))$, that is to say, $Y_t= \varprojlim Y^i_t$ for each $t$. $Y_t^i$ defined on $\mathcal{U}_i$ is a smooth function ($f= \varprojlim f_i$ is smooth if and only if each $f_i$ is smooth function \cite{galanis1}).

By the definition of compatibility of $\sigma$ with the inductive maps, we have that each $Y_t^i$ is $\mu$-Lipscitz.  We want to make sure that the isotopy of the time dependent vector field $Y_t$ exists. For this we will use theorem \ref{Flow}.    Comparing with the notation of the theorem \ref{thm:IFT}, $p_i(Y_t(x)):= \|Y_t^i(x)\|_i$.

Here  $\|.\|_i$ denotes norm on each $E_i$ and we know that  topology of $E=\varprojlim E_i$ is generated by countable norms $\{\|.\|_i: i\in \mathbb N\}$.

\begin{eqnarray*}
p_i(Y_t(x))&&= \|Y_t^i(x)\|_i\\
&&=\|(((\sigma_{x_i}^i)^t)^b)^{-1}(\alpha_{x_i})\|_i\\
&&\leq \|((\sigma_{x_i}^i)^{tb})^{-1}\|_{op}.\|\alpha^i_{x_i}\|_{op}\\
&&\leq \|((\sigma_{x_i}^i)^{tb})^{-1}\|_{op}.\|X_i(0)\|_{i}\\
&&\leq M
\end{eqnarray*}
We define $\tilde{Y}(x,t):= (Y_t(x), \frac{d}{dt})$ a vector field on $E\times \mathbb R$.  Vector field $\tilde{Y}$ satisfies the condition of the theorem \ref{Flow}. Therefore for each $t$, flow of time dependent vector field $Y_t$ exists.

We have given that $X(0)=0$ this gives $\tilde{Y}(0,t)= (0,\frac{d}{dt})$.  Therefore each flow is defined for all $t\in [-1,1]$ and there exists an isotopy $\phi_t$ for time dependent vector field $Y_t$.

We have:
$$\frac{d}{dt}\phi_t^*\sigma_t= \phi_t^*(L_{Y_t}\sigma_t)+\phi_t^*\frac{d}{dt}\sigma_t$$
$$=\phi_t^*(-d\alpha+\overline{\sigma})=0$$
Hence we have $\phi_1^*\sigma_1=\phi_0^*\sigma_0$.

This gives $$\phi_1^*\sigma= \sigma_0$$
\end{proof}
This proves existence of Darboux chart.


\begin{thebibliography}{99}
\bibitem{bambusi} D. Bambusi: On the Darboux theorem for weak symplectic manifolds. Porc. Amer. Math. Soc. 127 (1999), no. 11, 3383--3391.
\bibitem{galanis1}G. Galanis, Projective limits of Banach-Lie groups, Periodica Mathematica Hungarica 32 (1996), pp. 179--191.
\bibitem{galanis2} G. Galanis, On a type of linear differential equations in Fr\'{e}het spaces, Annali della Scuola Normale Superiore di Pisa, 4 No. 24 (1997), 501--510.
\bibitem{galanis3} G. Galanis, On a type of Fr\'{e}chet principal bundles, Periodica Mathematica Hungarica Vol. 35(1-2), 1997 pp. 15--30
\bibitem{galanis4} G. Galanis, P.K.Palamides, Nonlinear differential equations in Fr\'{e}chet spaces and continuum cross-sections,  Analele sthntifice ale universitatii Al.I. Cuza  IASI TOmul LI, s.I.Matematica, 2005 f.1
\bibitem{omri}H. Omori, Infinite Dimensional Lie Transformation Groups, Lecture Notes in Mathematics 427 (1974), Springer-Verlag.
\bibitem{marseden} J. E. Marsden: Darboux's Theorem Fails for Weak Symplectic Forms. Proc. Amer. Math. Soc. 32, (1972). MR 45:2755
\bibitem{weinstein}A. Weinstein: Symplectic Structures on Banach Manifolds. Bull. Amer. Math. Soc. 75, 1040-- 1041 (1969).
\bibitem{Leslie}J. A. Leslie, On a differential structure for the group of diffeomorphisms, Topology 46 (1967), 263--271
\bibitem{Hamilton}R. S. Hamilton. The Inverse Function Theorem of Nash and Moser. Bulletin of the American Mathematical Society, 7(1), July 1982.
\bibitem{KM} A. Kriegl and P. W. Michor. The Convenient Setting of Global Analysis, Mathematical Surveys and Monographs. American Mathematical Society, 1997.
\bibitem{galanis5} Differential and geometric structure for the tantent bundle of a projective limit manifold,  Rendiconti del seminario matematico dell' Universita di Padova  Vol 112  Anno 2004.
\bibitem{Cartan} H. Cartan,  Differential calculus (1971), Hermann Paris.
\end{thebibliography}
\end{document}